\documentclass[10pt]{amsart}
\usepackage{graphicx,amssymb}
\usepackage{a4wide}

\theoremstyle{plain}
\newtheorem{theorem}{Theorem}[section]
\newtheorem{lemma}[theorem]{Lemma}
\newtheorem{corollary}[theorem]{Corollary}
\newtheorem{proposition}[theorem]{Proposition}

\theoremstyle{definition}
\newtheorem{definition}[theorem]{Definition}
\newtheorem{example}[theorem]{Example}
\newtheorem{remark}[theorem]{Remark}

\def\arA{\mathbf A}
\def\arB{\mathbf B}
\def\arC{\mathbf C}

\def\Z{\mathbb Z}
\def\Q{\mathbb Q}

\def\myangle#1{\langle #1\rangle}

\def\len{\operatorname{len}}

\def\infs{\inf{\!}_ s}
\def\sups{\sup{\!}_s}
\def\lens{\len{\!}_s}

\def\t{\operatorname{\it t}}
\def\INF{\t_{\inf}}
\def\SUP{\t_{\sup}}
\def\LEN{\t_{\len}}

\let\epsilon\varepsilon

\begin{document}

\title{Notes on periodic elements of Garside groups}

\author{Eon-Kyung Lee and Sang-Jin Lee}
\address{Department of Mathematics, Sejong University,
    Seoul, 143-747, Korea}
\email{eonkyung@sejong.ac.kr}
\address{Department of Mathematics, Konkuk University,
    Seoul, 143-701, Korea}
\email{sangjin@konkuk.ac.kr}
\date{\today}

\begin{abstract}
Let $G$ be a Garside group with Garside element $\Delta$.
An element $g$  in $G$ is said to be \emph{periodic}
if some power of $g$ lies in the cyclic group generated by $\Delta$.
This paper shows the following.
(i) The periodicity of an element does not depend on
the choice of a particular Garside structure
if and only if the center of $G$ is cyclic.
(ii) If $g^k=\Delta^{ka}$ for some nonzero integer $k$,
then $g$ is conjugate to $\Delta^a$.
(iii)  Every finite subgroup of the quotient group
$G/\langle \Delta^m\rangle$ is cyclic,
where $\Delta^m$ is the minimal positive central power of $\Delta$.

\medskip\noindent
{\em Keywords\/}:
Garside group; periodic element; braid group.\\
{\em 2000 Mathematics Subject Classification\/}: Primary 20F36; Secondary 20F10\\
\end{abstract}

\maketitle

\section{Introduction}

The class of Garside groups, first introduced by Dehornoy and Paris~\cite{DP99},
is a lattice-theoretic generalization of braid groups
and Artin groups of finite type.
Recently, there were several results on periodic elements of Garside groups
such as the characterization of finite subgroups
of the central quotient of finite type Artin groups by Bestvina~\cite{Bes99}
and its extension to Garside groups by Charney, Meier and Whittlesey~\cite{CMW04};
a new algorithm for solving the conjugacy search problem for periodic braids
by Birman, Gebhardt and Gonz\'alez-Meneses~\cite{BGG07};
the characterization of periodic elements in the braid groups of
complex reflection groups by Bessis~\cite{Bes06a}.

This paper is interested in some questions on
periodic elements in Garside groups.
Let $G$ be a Garside group with Garside element $\Delta$.
An element $g$ in $G$ is said to be \emph{periodic}
(with respect to $\Delta$)
if either $g$ is the identity or it is commensurable with $\Delta$,
that is, $g^k$ and $\Delta^\ell$ are conjugate
for some nonzero integers $k$ and $\ell$~\cite{Bes06b, BGG08}.

\subsection{Periodicity and Garside structure}
The periodicity of an element in a Garside group
generally depends on the choice of a particular Garside structure,
more precisely on the Garside element.
A Garside group may admit more than one Garside structure.
For example, the $n$-strand braid group $B_n$
has two standard Garside structures:
the classical and dual Garside structures~\cite{Thu92,BKL98}.
Therefore it is natural to ask the following question.

\begin{quote}
When is the periodicity of an element
independent of the choice of a particular Garside structure?
\end{quote}

It is easy to see that the periodicity does not depend on the choice of a Garside structure
if and only if any two Garside elements
(arising from different Garside structures)
are commensurable,
and that this happens if the center $Z(G)$ is cyclic.
We show that the converse is also true.

\medskip\noindent
\textbf{Theorem~\ref{thm:Gar_struc}.}\ \ \em
Let $G$ be a Garside group.
Then $Z(G)$ is cyclic if and only if
any pair of Garside elements of $G$ are commensurable.
\upshape
\medskip

The irreducible Artin groups of finite type and, more generally,
the braid groups of irreducible well-generated complex reflection groups
are Garside groups with cyclic center~\cite{Bes06a}.
Therefore, in these groups, an element is periodic (with respect to a Garside element)
if and only if it has a central power.
However, not all Garside groups have cyclic center.
A typical example is $\Z^\ell$ for $\ell\ge 2$.

\subsection{$p/q$-periodicity, uniqueness of roots
and a question of Bessis}

First, we introduce a definition of Bessis in~\cite{Bes06b}:
an element $g\in G$ is $p/q$-periodic if $g^q=\Delta^p$ for $p\in\Z$ and $q\in\Z_{\ge 1}$.

Note that
$g^k=\Delta^\ell$ for some $k\in\Z_{\ge 1}$ and $\ell\in\Z$
if and only if
$g^q$ is conjugate to $\Delta^p$ for some $q\in\Z_{\ge 1}$ and $p\in\Z$,
because $\Delta^m$ is central for some $m\ge 1$.
Using this equivalence, we define the notion of $p/q$-periodicity
in a slightly different way.

\begin{definition}
An element $g\in G$ is said to be \emph{$p/q$-periodic}
for $p\in\Z$ and $q\in\Z_{\ge 1}$ if $g^q$ is conjugate to $\Delta^p$ and
$q$ is the smallest positive integer such that $g^q$
is conjugate to a power of $\Delta$.
\end{definition}

In the above definition, the $p/q$-periodicity \emph{a priori}
depends on the actual $p$ and $q$
and not just on the rational number $p/q$ because it may happen that
$g^{kq}$ is conjugate to $\Delta^{kp}$ for some $k\ge 2$
but $g^q$ is not conjugate to $\Delta^p$.
Motivated by this observation, we show the following.

\medskip\noindent
\textbf{Theorem~\ref{thm:unique}.}\ \ \em
Let $G$ be a Garside group with Garside element $\Delta$, and let $g\in G$
and $a, b, k \in\Z_{\neq 0}$.
\begin{itemize}
\item[(i)]
If\/ $g^{kb}$ is conjugate to $\Delta^{ka}$,
then $g^b$ is conjugate to $\Delta^a$.

\item[(ii)]
If\/ each of $g^{a}$ and $g^{b}$ is conjugate to a power of $\Delta$,
then so is $g^{\gcd(a,b)}$.
\end{itemize}
\upshape
\medskip

By this theorem, the term `$p/q$-periodic' contains that
$p$ and $q$ are coprime.

The above theorem is a sort of uniqueness property of roots up to conjugacy.
On this property, stronger results are known for some specific groups.
Let $g$ and $h$ be elements of a group $G$ such that
\begin{equation}\label{eq:root}
g^k=h^k\qquad\mbox{for some $k\ne 0$}.
\end{equation}
If $G$ is the pure $n$-braid group $P_n$, then
$g=h$ by Bardakov~\cite{Bar92}.
(This also follows from the biorderability of the pure braid groups
by Kim and Rolfsen~\cite{KR03}.)
If $G$ is the $n$-braid group $B_n$, then $g$ and $h$ are conjugate
by Gonz\'alez-Meneses~\cite{Gon03}.
If $G$ is the Artin group of type $\arB$, $\tilde \arA$ or $\tilde \arC$,
then $g$ and $h$ are conjugate~\cite{LL10}.
If $G$ is the braid group of a well-generated complex reflection group
and $g$ and $h$ are periodic elements,
then $g$ and $h$ are conjugate by Bessis~\cite{Bes06a}.
For a study of roots in mapping class groups, see~\cite{BP09}.

Theorem~\ref{thm:unique} shows that if $G$ is a Garside group
and $h$ is a power of $\Delta$, then (\ref{eq:root}) implies that
$g$ and $h$ are conjugate.
In Garside groups, even for periodic elements, it is hard to obtain a result
stronger than Theorem~\ref{thm:unique}.
For every $k\ge 2$, there is a Garside group with periodic elements $g$ and $h$
such that $g^k=h^k$ but $g$ and $h$ are not conjugate.
(See Example~\ref{ex:nonunique}.)

\medskip
The following is a question of Bessis~\cite[Question~4]{Bes06b}.

\smallskip\noindent\textbf{Question.}\ \
Let $G$ be a Garside group with Garside element $\Delta$.
Let $g\in G$ be a periodic element with respect to $\Delta$.
Does $G$ admit a Garside structure with Garside element $g$?
\smallskip

The above question is answered almost positively in the case of the braid group $B_n$:
each periodic element in $B_n$ is conjugate to
a power of one of the particular braids $\delta$ and $\epsilon$
which are the Garside elements in the dual Garside structures of $B_n$
and $A(\mathbf B_{n-1})$, respectively, where $A(\mathbf B_{n-1})$ denotes
the Artin group of type $\mathbf B_{n-1}$ viewed as a subgroup of $B_n$.
In~\cite{Bes06b}, Bessis showed that the above question is answered almost positively
in Garside groupoid setting.

To a Garside group $G$ with an affirmative answer to the above question,
the idea of Birman, Gebhardt and Gonz\'alez-Meneses in~\cite{BGG07}
can possibly be applied.
Precisely, in order to solve the conjugacy search problem for periodic elements
$g$ and $h$ of $G$,
it suffices to find a Garside structure with Garside element $g$.

Using Theorem~\ref{thm:unique}, we give a negative answer to the above question:
there is a Garside group $G$ with a periodic element $g$
such that there is no Garside structure on $G$ with Garside element $g$.
(See Example~\ref{ex:nonunique}.)

\subsection{Finite subgroups of the quotient group $G_\Delta$}
Let $m$ be the smallest positive integer such that $\Delta^m$ is central in $G$.
Let $G_\Delta$ be the quotient $G/\langle\Delta^m\rangle$,
where $\myangle{\Delta^m}$ is the cyclic group generated by $\Delta^m$.
For an element $g\in G$, let $\bar g$ denote the image of $g$
under the natural projection from $G$ to $G_\Delta$.
Hence, an element $g\in G$ is periodic if and only if
$\bar g$ has a finite order in $G_\Delta$.

About finite subgroups of $G_\Delta$, the following facts are known.

\begin{itemize}
\item[(i)]
If $G$ is an Artin group of finite type,
then every finite subgroup of\/ $G_\Delta$ is cyclic.

\item[(ii)]
If $G$ is a Garside group, then every finite subgroup of\/ $G_\Delta$
is abelian of rank at most 2.
\end{itemize}

The first was proved by Bestvina~\cite[Theorem 4.5]{Bes99}
and the second by
Charney, Meier and Whittlesey~\cite[Corollary 6.8]{CMW04}
following the arguments of Bestvina.
For the full statement of their results, see \S\ref{sec:central quotient}.

We show that Bestvina's result holds also for all Garside groups.

\medskip\noindent
\textbf{Theorem~\ref{thm:cyclic}.}\ \ \em
Let $G$ be a Garside group with Garside element $\Delta$.
Then every finite subgroup of\/ $G_\Delta$ is cyclic.
\upshape
\medskip

Our proof uses the result of Charney, Meier and Whittlesey.
Actually we prove that every finite abelian subgroup of $G_\Delta$ is cyclic.
Because every finite subgroup of $G_\Delta$ is abelian,
this implies the above theorem.

\section{Review of Garside groups}

This section describes basic definitions and properties of Garside groups.
See~\cite{DP99,Deh02} for details.

For a monoid $M$, let $1$ denote the identity element.
An element $a\in M\setminus \{ 1\}$ is called an \emph{atom} if
$a=bc$ for $b,c\in M$ implies either $b=1$ or $c=1$.
For $a\in M$, let $\Vert a\Vert$ be the supremum
of the lengths of all expressions of
$a$ in terms of atoms. The monoid $M$ is said to be \emph{atomic}
if it is generated by its atoms and $\Vert a\Vert<\infty$
for any element $a$ of $M$.
In an atomic monoid $M$, there are partial orders $\le_L$ and $\le_R$:
$a\le_L b$ if $ac=b$ for some $c\in M$;
$a\le_R b$ if $ca=b$ for some $c\in M$.

\begin{definition}
An atomic monoid $M$ is called a \emph{Garside monoid} if
\begin{enumerate}
\item[(i)] $M$ is left and right cancellative;
\item[(ii)] $(M,\le_L)$ and $(M,\le_R)$ are lattices;
\item[(iii)] $M$ contains an element $\Delta$, called a
\emph{Garside element}, satisfying the following:\\
(a) for each $a\in M$, $a\le_L\Delta$ if and only if $a\le_R\Delta$;\\
(b) the set $\{a\in M: a \le_L\Delta\}$ is finite and generates $M$.
\end{enumerate}
\end{definition}

An element $a$ of $M$ is called a \emph{simple element} if $a\le_L\Delta$.
A \emph{Garside group} is defined as the group of fractions
of a Garside monoid.
When $M$ is a Garside monoid and $G$ is the group of fractions of $M$,
we identify the elements of $M$ and their images in $G$
and call them \emph{positive elements} of $G$.
$M$ is called the \emph{positive monoid} of $G$,
often denoted by $G^+$.
The triple $(G, G^+, \Delta)$ is called a
\emph{Garside structure} on $G$.
We remark that a given group $G$ may
admit more than one Garside structure.

Let $\tau : G\to G$ be the inner automorphism of $G$
defined by $\tau(g)=\Delta^{-1}g\Delta$ for $g\in G$.
It is known that some power of $\tau$ is the identity,
equivalently, some power of $\Delta$ is central.

The partial orders $\le_L$ and $\le_R$, and thus the lattice structures
in the positive monoid $G^+$ can be extended to the Garside group $G$.
For $g, h\in G$, $g\le_L h$ (resp. $g\le_R h$) means $g^{-1}h\in G^+$
(resp. $hg^{-1}\in G^+$).

For $g\in G$, there are integers $r\le s$ such that
$\Delta^r\le_L g\le_L\Delta^s$.
Hence, the invariants
$\inf(g)=\max\{r\in\Z:\Delta^r\le_L g\}$,
$\sup(g)=\min\{s\in\Z:g\le_L \Delta^s\}$ and
$\len(g)=\sup(g)-\inf(g)$
are well-defined.

For $g\in G$, we denote its conjugacy class $\{ h^{-1}gh : h\in G\}$ by $[g]$.
Define $\infs(g)=\max\{\inf(h):h\in [g]\}$,
$\sups(g)=\min\{\sup(h):h\in [g]\}$ and $\lens(g)=\sups(g)-\infs(g)$.

For every element $g\in G$, the following limits are well-defined:
$$
\INF(g)=\lim_{n\to\infty}\frac{\inf(g^n)}n;\quad
\SUP(g)=\lim_{n\to\infty}\frac{\sup(g^n)}n;\quad
\LEN(g)=\lim_{n\to\infty}\frac{\len(g^n)}n.
$$
These limits were introduced in~\cite{LL07}
in studying translation numbers in Garside groups.
In this paper we will exploit the following properties.

\begin{proposition}[{\cite{LL07,LL08}}]\label{prop:tran}
For $g, h\in G$,
\begin{enumerate}
\item[(i)]
$\INF(h^{-1}gh)=\INF(g)$ and $\SUP(h^{-1}gh)=\SUP(g)$;

\item[(ii)]
$\INF(g^n)= n\cdot\INF(g)$ and $\SUP(g^n)= n\cdot\SUP(g)$
for all integers $n\ge 1$;

\item[(iii)]
$\infs(g)=\lfloor \INF(g)\rfloor$ and $\sups(g)=\lceil \SUP(g)\rceil$;

\item[(iv)]
$\INF(g)$ and $\SUP(g)$ are rational of the form $p/q$,
where $p$ and $q$ are coprime integers and
$1\le q\le\Vert\Delta\Vert$.
\end{enumerate}
\end{proposition}

\section{Periodicity and Garside structure}\label{sec:Garsid structure}

For a group $G$, let $Z(G)$ denote the center of $G$.
The following is the main result of this section.

\begin{theorem}\label{thm:Gar_struc}
Let $G$ be a Garside group.
Then $Z(G)$ is cyclic if and only if
any pair of Garside elements of $G$ are commensurable.
\end{theorem}

It is easy to see that
the periodicity of an element in $G$ does not depend on
the choice of a particular Garside element
if and only if any pair of Garside elements of $G$ are commensurable.
Hence we have the following.

\begin{corollary}\label{cor:Gar_struc}
Let $G$ be a Garside group.
The periodicity of an element of $G$ does not depend on the choice of a particular
Garside element if and only if $Z(G)$ is cyclic.
\end{corollary}

\smallskip
We prove Theorem~\ref{thm:Gar_struc} by using the following lemma.
For $g\in G$, let
$L(g)=\{ a\in G^{+} : a\le_L g \}$ and $R(g)=\{ a\in G^{+} : a\le_R g \}$.

\begin{lemma}\label{lem:Gar_struc}
Let $(G, G^+, \Delta)$ be a Garside structure on a group $G$.
\begin{itemize}
\item[(i)]
Let $c$ be a positive element in $Z(G)$.
Then $L(c)=R(c)$.

\item[(ii)]
Let $c$ be a positive element in $Z(G)$
such that $\Delta\le_L c$.
Then $c$ is a Garside element, that is,
$L(c)=R(c)$ and $L(c)$ generates the positive monoid $G^+$.
\end{itemize}
\end{lemma}

\begin{proof}
(i)\ \
Let $a\in L(c)$, then $c=ab$ for some $b\in G^+$.
Because $c$ is central, $ab=c=bcb^{-1}=b(ab)b^{-1}=ba$.
Therefore $c=ba$, hence $a\in R(c)$.
This means that $L(c)\subset R(c)$.
Similarly, $R(c)\subset L(c)$.

\smallskip(ii)\ \
By (i), $L(c)=R(c)$.
As $\Delta\le_L c$, we have $L(\Delta)\subset L(c)$.
Since $L(\Delta)$ generates $G^+$, so does $L(c)$.
\end{proof}

\begin{proof}[Proof of Theorem~\ref{thm:Gar_struc}]
Suppose that $Z(G)$ is cyclic.
Let $(G, G_1^+, \Delta_1)$ and $(G, G_2^+, \Delta_2)$ be Garside structures on $G$.
Then there exist positive integers $m_1$ and $m_2$ such that
$\Delta_1^{m_1}$ and $\Delta_2^{m_2}$ are central in $G$.
Because $Z(G)$ is cyclic,
$\Delta_1^{m_1}$ and $\Delta_2^{m_2}$ are commensurable,
hence $\Delta_1$ and $\Delta_2$ are commensurable.

\smallskip
Conversely, suppose any pair of Garside elements are commensurable.
Fix a Garside structure $(G, G^+, \Delta)$ on $G$.
Let $m$ be the smallest positive integer such that $\Delta^m$ is central.

We claim that any nonidentity element of $Z(G)$ is commensurable with $\Delta$.
Let $g$ be a nonidentity central element.
Take an integer $k$ such that
$$
k\equiv 0\bmod m\quad\mbox{and}\quad k\ge -\inf(g)+1.
$$
Let $c=\Delta^k g$. Then $c$ is a central element with $\Delta\le_L c$,
hence $c$ is a Garside element by Lemma~\ref{lem:Gar_struc}.
By the hypothesis, $c$ is commensurable with $\Delta$.
As $c$ is central, there exist nonzero integers $p$ and $q$ such that
$\Delta^p=c^q=(\Delta^kg)^{q} = \Delta^{kq}g^q$.
Since $g^q=\Delta^{p-kq}$, $g$ is commensurable with $\Delta$.

It is known that Garside groups are torsion-free
by Dehornoy~\cite{Deh98},
and that every abelian subgroup of a Garside group is
finitely generated by Charney, Meier and Whittlesey~\cite{CMW04}.
Thus $Z(G)$ is torsion-free and finitely generated.
Moreover, by the above claim, any two nonidentity elements of $Z(G)$ are commensurable
because each of them is commensurable with $\Delta$.
These imply that $Z(G)$ is cyclic.
\end{proof}

\section{Roots of periodic elements}\label{sec:roots}

Let $G$ be a Garside group with Garside element $\Delta$.
First, we explore some basic properties of periodic elements regarding translation numbers.

\begin{lemma}\label{lem:per-elt}
Let $g\in G$ be a periodic element. Then the following hold.
\begin{itemize}
\item[(i)] $\LEN(g)=0$, that is, $\INF(g)=\SUP(g)$.
\item[(ii)] $\INF(g^k) = k\cdot\INF(g)$ for all $k\in\Z$.
\item[(iii)]
For any $k\in\Z$, $\lens(g^k)$ is either 0 or 1.

\item[(iv)]
Let $\INF(g)=p/q$ for $p\in\Z$ and $q\in\Z_{\ge 1}$ with $\gcd(p,q)=1$.
Then the following are equivalent for $k\in\Z$:
\begin{itemize}
\item[(a)] $g^k$ is conjugate to a power of $\Delta$;
\item[(b)] $\lens(g^k)=0$;
\item[(c)] $\INF(g^k)$ is an integer;
\item[(d)] $k$ is a multiple of $q$.
\end{itemize}
In particular, $g^q$ is conjugate to $\Delta^p$.
\end{itemize}
\end{lemma}

\begin{proof}
(i)\ \
Since $g^k=\Delta^\ell$ for some $k\in\Z_{\ge 1}$ and $\ell\in\Z$,
$$
\LEN(g)=\frac1k\cdot\LEN(g^k)=\frac1k\cdot\LEN(\Delta^\ell)
=\frac1k\cdot 0=0.
$$

\smallskip(ii)\ \
We know that $\INF(g^k) = k\cdot\INF(g)$ holds for all $k\ge 0$.
Let $k<0$, then $k=-\ell$ for some $\ell\ge 1$.
Since $\INF(g)=\SUP(g)$ by (i)
and $\INF(h^{-1})=-\SUP(h)$ for all $h\in G$,
we have
$$
\INF(g^k) =\INF((g^{-1})^\ell)
=\ell\cdot\INF(g^{-1})
=\ell\cdot(-\SUP(g))=k\cdot\INF(g).
$$

\smallskip(iii) \ \
Let $\INF(g)=p/q$ for $p\in\Z$ and $q\in\Z_{\ge 1}$ with $\gcd(p,q)=1$.
Choose any $k\in\Z$.
Because $\SUP(g)=\INF(g)=p/q$ by (i),
\begin{equation}
\label{eq:lens}
\lens(g^k)
=\sups(g^k)-\infs(g^k)
=\lceil \SUP(g^k)\rceil -\lfloor\INF(g^k)\rfloor
=\lceil kp/q\rceil -\lfloor kp/q\rfloor
\end{equation}
by (ii). Therefore $\lens(g^k)$ is either 0 or 1.

\smallskip (iv) \ \
It is obvious that $\lens(g^k)=0$ if and only if $g^k$ is conjugate
to a power of $\Delta$.
By Eq.~(\ref{eq:lens}),
$\lens(g^k)=0$
if and only if $\INF(g^k)=kp/q$ is an integer.
Because $p$ and $q$ are coprime,
$kp/q$ is an integer if and only if $k$ is a multiple of $q$.
Therefore the four conditions---(a), (b), (c) and (d)---are equivalent.

By (a) and (d), $g^q$ is conjugate to $\Delta^\ell$ for some
integer $\ell$, hence $\INF(g)=\ell /q$.
Because $\INF(g)=p/q$ by the hypothesis, we have $\ell=p$.
\end{proof}

Theorem 5.1 in \cite{LL08} shows that for every $g\in G$,
$$\begin{array}{rcl}
\INF(g)&=&\max\{\infs(g^k)/k : k=1,\ldots, \Vert\Delta\Vert\},\\
\SUP(g)&=&\min\{\sups(g^k)/k : k=1,\ldots, \Vert\Delta\Vert\}.
\end{array}
$$
Hence the values of $\INF(g)$, $\SUP(g)$ and $\LEN(g)$
can be computed explicitly.

\begin{remark}
If $g$ is an element of $G$ with $\LEN(g)=0$, then $\INF(g)=\SUP(g)=p/q$
for some $p\in\Z$ and $q\in\Z_{\ge 1}$.
Hence $\infs(g^q)=\sups(g^q)=p$, which implies that $g$ is periodic.
Therefore we can see that $g$ is periodic if and only if $\LEN(g)=0$.
Further, $g$ is $p/q$-periodic if and only if $\INF(g)=\SUP(g)=p/q$
for $p\in\Z$, $q\in\Z_{\ge 1}$ with $\gcd(p,q)=1$.
\end{remark}

\smallskip

The following is the main result of this section,
which is a consequence of Lemma~\ref{lem:per-elt}(iv).

\begin{theorem}\label{thm:unique}
Let $G$ be a Garside group with Garside element $\Delta$, and let $g\in G$
and $a, b, k \in\Z_{\neq 0}$.
\begin{itemize}
\item[(i)]
If\/ $g^{kb}$ is conjugate to $\Delta^{ka}$,
then $g^b$ is conjugate to $\Delta^a$.

\item[(ii)]
If\/ each of $g^{a}$ and $g^{b}$ is conjugate to a power of $\Delta$,
then so is $g^{\gcd(a,b)}$.
\end{itemize}
\end{theorem}

\begin{proof}
The hypothesis in either case of (i) or (ii) implies that $g$ is periodic.
Let $\INF(g)=p/q$ for $p\in\Z$ and $q\in\Z_{\ge 1}$ with $\gcd(p,q)=1$.
Then $g^{q}$ is conjugate to $\Delta^{p}$ by Lemma~\ref{lem:per-elt}(iv).

\smallskip

(i)\ \
Since $g^{kb}$ is conjugate to $\Delta^{ka}$, one has
$\INF(g)=a/b=p/q$, hence there is $d\in\Z_{\neq 0}$ such that $a=dp$ and $b=dq$.
Therefore $g^b$ is conjugate to $\Delta^a$.

\smallskip

(ii)\ \
By Lemma~\ref{lem:per-elt}(iv), both $a$ and $b$ are multiples of $q$,
hence $\gcd(a,b)$ is a multiple of $q$.
Therefore $g^{\gcd(a,b)}$ is conjugate to some power of\/ $\Delta$.
\end{proof}

Theorem~\ref{thm:unique}(i) implies that
if $g^k=\Delta^{k\ell}$ then $g$ is conjugate to $\Delta^\ell$.
The following example illustrates that
the general statement for the uniqueness of roots up to conjugacy
(i.e. if $g^k = h^k$ for $k\neq 0$ then $g$ is conjugate to $h$)
does not hold in Garside groups, even for periodic elements $g$ and $h$.

\begin{example}\label{ex:nonunique}
Let $G$ be the group defined by
$$
G=\langle x,y\mid x^a=y^a\rangle,\qquad a\ge 2.
$$
It is a Garside group with Garside element
$\Delta=x^a=y^a$~\cite[Example 4]{DP99}.
(Note that $x$ and $y$ are periodic elements and that $\Delta$ is central.)
We claim that
\begin{itemize}
\item[(i)] $x$ and $y$ are not conjugate;
\item[(ii)] there is no Garside structure on $G$
in which $x$ is a Garside element.
\end{itemize}

Because $G/\langle \Delta\rangle =\langle x,y\mid x^a=y^a=1\rangle
=\langle x\mid x^a=1\rangle*\langle y\mid y^a=1\rangle$,
the images of $x$ and $y$ in $G/\langle \Delta\rangle$ are not conjugate.
Therefore $x$ and $y$ are not conjugate.

Assume that there exists a Garside structure on $G$ with Garside element $x$.
Because $x^a=y^a$ and $x$ is a Garside element,
$y$ is conjugate to $x$ by Theorem~\ref{thm:unique}(i).
It is a contradiction to (i).
\end{example}

The above example shows that
\begin{quote}
there is a Garside group with a periodic element $g$
such that there is no Garside structure in which
$g$ is a Garside element.
\end{quote}
Therefore it gives a negative answer to the question of Bessis
stated in \S1.

\section{Finite subgroups of the quotient group $G_\Delta$}\label{sec:central quotient}

Let $G$ be a Garside group with Garside element $\Delta$,
and let $m$ be the smallest positive integer such that
$\Delta^m$ is central in $G$.
Let $G_\Delta$ denote the quotient $G/\myangle{\Delta^m}$.
For an element $g\in G$, let $\bar g$ denote the image of $g$
under the natural projection from $G$ to $G_\Delta$.
The following theorem was proved
by Bestvina~\cite[Theorem 4.5]{Bes99} for
Artin groups of finite type, and then
proved by Charney, Meier and Whittlesey~\cite[Corollary 6.8]{CMW04}
for Garside groups.

\begin{theorem}[\cite{Bes99,CMW04}]
\label{thm:CMW}
The finite subgroups of\/ $G_\Delta$ are, up to conjugacy,
one of the following two types:
\begin{itemize}
\item[(i)]
the cyclic group generated by the image of $\Delta^u a$ in $G_\Delta$
for some $u\in\Z$ and some simple element $a\ne\Delta$
such that if $a\neq 1$, then for some integer $2\le q\le\Vert\Delta\Vert$
$$
\tau^{(q-1)u}(a)\,\tau^{(q-2)u}(a)\cdots \tau^{u}(a)\,a=\Delta ;
$$

\item[(ii)]
the direct product of a cyclic group of type (i) and
$\langle\bar\Delta^k\rangle$ where $\Delta^k$ commutes with $a$.
\end{itemize}
\end{theorem}

In the case of Artin groups of finite type,
Bestvina showed that finite subgroups of $G_\Delta$ are all cyclic groups
(hence they are of type (i) in the above theorem).
Using the following lemma, we show in Theorem~\ref{thm:cyclic}
that the same is true for Garside groups.

\begin{lemma}\label{lem:tinf}
Let $H$ be an abelian subgroup of $G$ which consists of periodic elements.
Then  $\INF|_{H}: H\to \Q$ is a monomorphism.
In particular, $H$ is a cyclic group.
\end{lemma}

\begin{proof}
Let $h_1,h_2\in H$ with $\INF(h_i)=p_i/q_i$ for $i=1,2$.
Because $h_i^{q_i}$ is conjugate to $\Delta^{p_i}$ (by Lemma~\ref{lem:per-elt})
and $\Delta^m$ is central, one has $h_i^{q_im}=\Delta^{p_im}$ for $i=1,2$.
Therefore
$$
(h_1h_2)^{q_1q_2m}
=(h_1^{q_1m})^{q_2} \cdot (h_2^{q_2m})^{q_1}
=\Delta^{p_1mq_2}\Delta^{p_2mq_1}
=\Delta^{m(p_1q_2+p_2q_1)},
$$
hence
$\INF(h_1h_2) = m(p_1q_2+p_2q_1) / q_1q_2m = p_1/q_1+p_2/q_2 =\INF(h_1)+\INF(h_2)$.
This means that $\INF|_{H}: H\to \Q$ is a homomorphism.
If $h\in H$ and $\INF(h)=0$, then
$h$ is conjugate to $\Delta^0=1$ by Lemma~\ref{lem:per-elt}, hence
$h=1$.
This means that $\INF|_{H}: H\to \Q$ is injective.

Notice that, for all $g\in G$, $\INF(g)$ is rational of the form $p/q$ with $1\le q\le\Vert\Delta\Vert$
(see Proposition~\ref{prop:tran}).
Therefore $\INF(H)$ is a discrete subgroup of $\Q$, hence it is a cyclic group.
Because  $\INF|_{H}: H\to \Q$ is injective,
$H$ is also a cyclic group.
\end{proof}

\begin{theorem}\label{thm:cyclic}
Let $G$ be a Garside group with Garside element $\Delta$.
Then every finite subgroup of\/ $G_\Delta$ is cyclic.
\end{theorem}

\begin{proof}
Let $K$ be a finite subgroup of $G_\Delta$.
Let $H$ be the preimage of $K$ under the natural projection $G\to G_\Delta$.
Notice that, by Theorem~\ref{thm:CMW},
$H$ is abelian and every element of $H$ is periodic.
By Lemma~\ref{lem:tinf}, $H$ is a cyclic group,
hence $K$ is cyclic.
\end{proof}

\end{document}